\documentclass[11pt]{article}
\usepackage[textwidth=165mm,textheight=210mm]{geometry}
\usepackage{amsmath}
\usepackage{amssymb}
\usepackage{amsthm}
\usepackage{graphicx}
\usepackage{microtype}
\usepackage{hyperref}
\usepackage{enumitem}

\newtheorem{theorem}{Theorem}[section]
\newtheorem{lemma}[theorem]{Lemma}

\newtheorem{corollary}[theorem]{Corollary}
\theoremstyle{definition}

\newtheorem{definition}[theorem]{Definition}

\newtheorem{remark}[theorem]{Remark}
\newtheorem{notation}[theorem]{Notation}

\def\x{\xi}
\def\p{\partial}

\def\W{\mathcal{W}}

\def\W{\mathcal{W}}

\makeatletter
\newcommand{\rmnum}[1]{\romannumeral #1}
\newcommand{\Rmnum}[1]{\expandafter\@slowromancap\romannumeral #1@}
\makeatother

\begin{document}
\title {\textbf{Intersection numbers in the curve graph with a uniform constant}}
\author{Yohsuke Watanabe\thanks{The author was partially supported from  U.S. National Science Foundation grants DMS 1107452, 1107263,
1107367 ``RNMS: Geometric Structures and Representation Varieties'' (the GEAR Network).}}
\date{}
\maketitle
\begin{abstract}
We derive various inequalities involving the intersection number of the curves contained in geodesics and tight geodesics in the curve graph. 
While there already exist such inequalities on tight geodesics, our method applies in the setting of geodesics. Furthermore, the method gives inequalities with a uniform constant depending only on the topology of the surface.\end{abstract}



\section{Introduction}
Let $S_{g,n}$ be a compact surface of genus $g$ and $n$ boundary components. Throughout this paper, we assume that an isotopy is free unless otherwise specified and that curves are simple, closed, essential and not isotopic to $\partial(S)$. 
We recall the curve graph, $C(S)$ defined by Harvey \cite{HAR}. Suppose $\x(S)=3g+n-3\geq 1.$ The vertices are isotopy classes of curves and the edge between two vertices are realized by disjointness. We manipulate the definition of an edge for $\x(S)=1$; we put the edge between two vertices if they intersect once if $S=S_{1,1}$ and twice if $S=S_{0,4}$. The curve graph is a geodesic metric space with the usual graph metric (which assigns distance $1$ to each edge.), which we denote by $d_{S}$. 

\begin{definition} Let $x,y\in C(S)$ and $A,B\subseteq C(S)$.
\begin{itemize}
\item The intersection number between $x$ and $y$ is the minimal possible number of intersections between them up to isotopy, and we denote it by $i(x,y)$. We define $\displaystyle i(A,B):=\sum_{a\in A,b\in B}i(a,b).$

\item The distance between $x$ and $y$ is the length of a geodesic between $x$ and $y$, and we denote it by $d_{S}(x,y)$. We define $d_{S}(A,B):=diam_{C(S)}A\cup B .$

\item We say $A$ and $B$ fill $S$ if $i(c,A)>0$ or $i(c,B)>0$ for all $c\in C(S)$. If $\x(S)>1$, then $A$ and $B$ fill $S$ if and only if $d_{S}(A,B)\geq 3$. Lastly, we let $F(A,B)$ denote a regular neighborhood of $A\cup B$ in $S$.
\end{itemize}
\end{definition}

We recall the definition of tight (multi)geodesics defined by Masur--Minsky \cite{MM2}. Note that a tight geodesic always exists between any pair of curves \cite{MM2}.

\begin{definition}
\begin{itemize}
\item Suppose $\x(S)=1$. Every geodesic is defined to be a tight geodesic. 

\item Suppose $\x(S)>1$. 
A multicurve is a set of mutually disjoint curves in $S$. A multigeodesic is a sequence of multicurves $\{V_{i}\}$ such that $d_{S}(a,b)=|p-q|$ for all $a\in V_{p},b\in V_{q}$ and for all $p,q$. 
A tight multigeodesic is a multigeodesic $\{V_{i}\}$ such that $V_{i}=\partial (F(V_{i-1},V_{i+1}))$ for all $i$. 
Given $x,y\in C(S)$, a tight geodesic between $x$ and $y$ is a geodesic $\{v_{i}\}$ such that $v_{i}\in V_{i}$ for all $i$ where $\{V_{i}\}$ is a tight multigeodesic between $x$ and $y$. 
\end{itemize}
\end{definition}

In this paper, we study the intersection numbers of the curves which are contained in geodesics (Theorem \ref{E''}) and tight geodesics (Theorem \ref{E'}). We review some works related to this paper. For the rest of this paper, we let $g_{x,y}$ denote a (multi)geodesic between $x$ and $y$ in $C(S);$ we will always specify whether it is tight or not.

Shackleton showed

\begin{theorem}[{\cite[Theorem 1.1]{SHA}}]\label{tabi}
Suppose $\x(S)>1$. Let $x,y\in C(S)$ and $g_{x,y} = \{v_{i}\}$ be a tight multigeodesic such that $d_{S}(x,v_{i}) = i$ for all $i$. Let $F(n)=n \cdot T^{\lfloor 2 \log_{2} n \rfloor}$ where $T$ depends only on the surface. Then $ i(v_{p},y) \leq \underbrace{F\circ F\circ  \cdots \circ F}_{p \text{ many } F \text{'s} }(i(x,y))$ for all $p.$
\end{theorem}

The author showed

\begin{theorem}[{\cite[Theorem 1.6]{W3}}]\label{tabi2}
In Theorem \ref{tabi}, $F(n)$ can be replaced by a linear function $F(n)=R\cdot n$ where $R>1$ depends only on the surface, and we have $i(v_{p},y) \leq R^{p}\cdot i(x,y)$ for all $p.$ 
\end{theorem}


We note that the final constants which appeared in the inequalities in Theorem \ref{tabi} and Theorem \ref{tabi2} depend not only on the surface but also on $d_{S}(x,v_{p})$ and $d_{S}(x,y)$. For instance, if $h=\big\lfloor \frac{d_{S}(x,y)}{2} \big \rfloor$, then in Theorem \ref{tabi2} we have $i(v_{h},y) \leq R^{h}\cdot i(x,y).$ In particular, if $d_{S}(x,y) \rightarrow \infty$ then $R^{h}\rightarrow \infty.$
Our main contribution of this paper is to overcome this issue; we derive an inequality on tight geodesics, where the constant depends only on the surface. Furthermore, this technique applies to obtain a such inequality in the setting of geodesics. 
Lastly, we remark that Theorem \ref{tabi2} also holds when $\x(S)=1$ with $R\leq1$, see \cite{W3}. Hence, for the rest of this paper, we always assume $\x(S)>1.$

We show 
\begin{theorem}\label{E'}
Let $x,y\in C(S)$ and $g_{x,y} = \{v_{i}\}$ be a tight geodesic such that $d_{S}(x,v_{i}) = i$ for all $i$. There exists $U$ depending only on $S$ such that 
$i(v_{p},v_{q}) \leq i(x,y)^{U}$ for all $p,q$.
\end{theorem}


We notice that Theorem \ref{E'} is analogous to an obvious inequality involving the distance of the curves contained in a geodesic: $$d_{S}(v_{p},v_{q})\leq d_{S}(x,y)\text{ for all }p,q.$$ In other words, the intersection number of the curves contained in a tight geodesic behaves like the distance of the curves contained in a geodesic with an extra uniform constant on its exponent.
We show a converse version, Corollary \ref{Direct}, which is analogous to an another obvious inequality involving the distance of the curves contained in a geodesic: $$d_{S}(x,y) \leq d_{S}(x,v_{p})+d_{S}(v_{p},y)\text{ for all }p.$$

Corollary \ref{Direct} directly follows from the following theorem.

 \begin{theorem}\label{E''}
Let $x,y\in C(S)$ and $g_{x,y} = \{v_{i}\}$ be a geodesic such that $d_{S}(x,v_{i}) = i$ for all $i$. There exists $V$ depending only on $S$ such that $ i(x,y)\leq \big( i(x,v_{p})\cdot i(v_{p},y)\cdot i(x,v_{q})\cdot i(v_{q},y)\big)^{V}$ for all $p,q$ such that $|p-q|>2$. (We let $i(x,v_{1})=1$ and $i(v_{d_{S}(x,y)-1},y)=1$.)
\end{theorem}

We note that $g_{x,y}$ does not have to be tight and the length of $g_{x,y}$ needs to be at least 5 in the above. Also, we let $i(x,v_{1})=1$ and $i(v_{d_{S}(x,y)-1},y)=1$ even though these intersection numbers are $0$; in Remark \ref{hkai}, we explain the reason that the above theorem also holds with this modification.

In the rest of this section, $U$ and $V$ will denote the constants given in Theorem \ref{E'} and Theorem \ref{E''} respectively. 

By Theorem \ref{E''}, we have 
\begin{corollary}\label{Direct}
Let $x,y\in C(S)$ and $g_{x,y} = \{v_{i}\}$ be a geodesic such that $d_{S}(x,v_{i}) = i$ for all $i$. Except for at most $3$ consecutive vertices of $g_{x,y}$, we have $i(x,y)\leq \big( i(x,v_{p})\cdot i(v_{p},y)\big)^{2V}$. 
\end{corollary}
\begin{proof}
Since $ i(x,y)\leq \big( i(x,v_{p})\cdot i(v_{p},y)\cdot i(x,v_{q})\cdot i(v_{q},y)\big)^{V}$ for all $p,q$ such that $|p-q|>2$ by Theorem \ref{E''}, we have $ i(x,y)\leq \big( i(x,v_{p})\cdot i(v_{p},y)\big)^{2V}$ or $ i(x,y)\leq \big( i(x,v_{q})\cdot i(v_{q},y)\big)^{2V}$ for all $p,q$ such that $|p-q|>2$. 
\end{proof}
With Theorem \ref{E'}, Theorem \ref{E''} and Corollary \ref{Direct}, we have

\begin{corollary}\label{E'''}
Let $x,y\in C(S)$ and $g_{x,y} = \{v_{i}\}$ be a tight geodesic such that $d_{S}(x,v_{i}) = i$ for all $i$. We have $$ \sqrt[4U]{i(x,v_{p})\cdot i(v_{p},y)\cdot i(x,v_{q})\cdot i(v_{q},y)} \leq i(x,y)\leq   \big( i(x,v_{p})\cdot i(v_{p},y)\cdot i(x,v_{q})\cdot i(v_{q},y)\big)^{V}.$$ 
for all $p,q$ such that $|p-q|>2$. 
Furthermore, except for at most $3$ consecutive vertices of $g_{x,y}$, we have $$ \sqrt[2U]{i(x,v_{p})\cdot i(v_{p},y)}\leq  i(x,y)\leq   \big( i(x,v_{p})\cdot i(v_{p},y)\big)^{2V}.$$ 

\end{corollary}

\textbf{Plan of the paper.}
A key technique to derive the main results stated in $\S 1$ is to use Theorem \ref{E} by varying a constant $n$ in the statement, see $\S 3$. Theorem \ref{E'} and Theorem \ref{E''} easily follow from this technique with Lemma \ref{sss} (it requires some work to show, but straightforward) and Remark \ref{r} (an elementary fact) respectively. Therefore, $\S 2$ will be the key section as we develop technical machinery to be used in $\S 3$.
The main result of $\S 2$ is Theorem \ref{E}. 
We will prove Theorem \ref{E} by using Choi--Rafi formula, Theorem \ref{marking}; nevertheless, the difficulty in this approach is to construct a pair of markings from a given pair of curves controlling intersection numbers, see Lemma \ref{E1}, Lemma \ref{E2}, and Corollary \ref{E3}. 
However, once we have Corollary \ref{E3}, Theorem \ref{E} immediately follows from Choi--Rafi formula with some elementary observations. 


\renewcommand{\abstractname}{\textbf{Acknowledgements}}
\begin{abstract}
The author thanks Kenneth Bromberg for useful discussion on $\S 2.1.1.$ The author also thanks Mladen Bestvina and Kasra Rafi for useful conversations. Some parts of this paper was written during his stay at the University of Illinois at Urbana--Champaign, the author thanks Christopher Leininger for his warm hospitality and the GEAR Network for supporting this trip. Lastly, the author thanks the referee for useful suggestions in the revision process.
 \end{abstract}

\section{\textbf{Background and machinery}}
The main goal of this section is to obtain Lemma \ref{sss} and Theorem \ref{E}. The proofs of Theorem \ref{E'} and Theorem \ref{E''} rely on Lemma \ref{sss} and Theorem \ref{E}.

First, we briefly review our basic tool, \emph{subsurface projections}. For a detailed treatment, see \cite{MM2}.
Let $Z$ be a subsurface of $S$. The subsurface projection is a map $$\pi_{Z}:C(S)\longrightarrow C(Z).$$

Suppose $Z$ is not an annulus. If $x\in C(S)$, then $\pi_{Z}(x)$ is a curve in $Z$ which is obtained by first picking an arc or a curve $a\in \{x \cap Z\}$ and taking a boundary component of a regular neighborhood of $a\cup \p(Z)$ in $Z$. 

Suppose $Z$ is an annulus. Fix a hyperbolic metric on $S$ and compactify the annular cover of $S$, which corresponds to $Z$, with its Gromov boundary; we denote the resulting cover by $S^{Z}$. We define the annular--curve graph of $Z$ on $S^{Z}$, altering the original definition given in $\S 1$; the vertices are the set of isotopy classes of arcs which connect two boundary components of $S^{Z}$, here the isotopy is relative to $\partial (S^{Z}) $ \emph{pointwise}. We put the edge between two vertices if they can be realized disjointly in the interior of $S^{Z}$. If $x \in C(S)$, then $\pi_{Z}(x)$ is an arc obtained by the lift of $x$ which connects two boundary components of $S^{Z}$. 

Let $A,B \subseteq C(S)$. For both non--annular and annular projections, we define $ \displaystyle \pi_{Z}(A):=\bigcup_{a\in A}\pi_{Z}(a)$ and $ d_{Z}(A,B):=diam_{C(Z)} \pi_{Z}(A) \cup \pi_{Z}(B).$

\begin{remark}\label{r}

We remark that if $A,B \subseteq C(S)$ such that $d_{S}(A,B)\geq 3$ then $\pi_{Z}(A)\neq \emptyset$ or $\pi_{Z}(B)\neq \emptyset$ because $A$ and $B$ fill $S$.

\end{remark}



We recall the following results from \cite{MM2}.

\begin{lemma}[{\cite[Lemma 2.2 \& Lemma 2.3]{MM2}}]\label{oct}
If $x,y\in C(S)$ such that $d_{S}(x,y)\leq 1$ then $d_{Z}(x,y)\leq 2$ for all $Z\subseteq S$.
\end{lemma}
We note that $2$ in Lemma \ref{oct} needs to be replaced by $3$ if $\x(S)=1$. The following theorem is called the \emph{Bounded Geodesic Image Theorem}.
\begin{theorem}[{\cite[Theorem 3.1]{MM2}}] \label{BGIT}
Let $Z$ be a proper subsurface of $S$. Let $\{v_{i}\}_{0}^{n}$ be a (multi)geodesic in $C(S)$ such that $\pi_{Z}(v_{i})\neq \emptyset$ for all $ i $. There exists $M$ depending only on $S$ so that $d_{Z} (v_{0},v_{n}) \leq M.$
\end{theorem}
In the rest of this paper, $M$ will denote the constant given by Theorem \ref{BGIT}. Note that $M$ can be taken so that it does not depend on $S$, for instance $M\leq 200$, see \cite{WEB1}.

We observe a special behavior of tight geodesics under the Bounded Geodesic Image Theorem. 

\begin{lemma}\label{sss}
Let $x,y\in C(S)$ and $g_{x,y}=\{v_{i}\}$ be a tight geodesic such that $d_{S}(x,v_{i})=i$ for all $i$. 
Suppose $\pi_{Z}(v_{p})\neq \emptyset$ and $\pi_{Z}(v_{q})\neq \emptyset$ where $Z\subsetneq S$. Assume $q>p$. If $d_{Z}(v_{p},v_{q})> M$ then $d_{Z}(x,v_{p})\leq M \text{ and }d_{Z}(v_{q},y)\leq M.$

\end{lemma}
\begin{proof}
Take a tight multigeodesic $\{V_{i}\}$ between $x$ and $y$ such that $v_{i} \in V_{i}$ for all $i$. If $d_{Z}(v_{p},v_{q})> M$ then we must have $\pi_{Z}(V_{h})= \emptyset$ where $p<h<q$ by Theorem \ref{BGIT}.

If $p+1<h<q-1$, then $\pi_{Z}(V_{k})\neq \emptyset$ for all $k<p$ and for all $k>q$ because $V_{h}$ and $V_{k}$ fill $S$; we are done by Theorem \ref{BGIT}. 

If $h=p+1$ or $h=q-1$, then we use tightness. Assume $h=p+1$. Since $V_{p}=\partial (F(V_{p-1},V_{p+1}))$, $\pi_{Z}(V_{p+1})=\emptyset$ and $\pi_{Z}(V_{p})\neq \emptyset$, we must have $\pi_{Z}(V_{p-1})\neq \emptyset$. We repeat the argument in the previous case about filling, and conclude $\pi_{Z}(V_{k})\neq \emptyset$ for all $k<p$. By Theorem \ref{BGIT}, we have $d_{Z}(x, v_{p})\leq M$. Lastly, by using similar techniques given so far, we also observe $d_{Z}(v_{q},y)\leq M$.


\end{proof}
\subsection{On Choi--Rafi formula}
A \emph{Pants decomposition} is a collection of mutually disjoint curves which cut the surface into pairs of pants. A \emph{marking} is a collection of curves obtained by taking a pants decomposition and choosing extra curves so that they together fill the surface. We call such extra curves \emph{transversal curves}.
For the rest of this paper, we use the following notations.

\begin{notation}
Let $n,m\in \mathbb{R}$, $n\prec m$ means there exists positive constants $k,c$ such that $n\leq k\cdot m+c$. If $n\prec m$ and $m\prec n$ then we write $n\asymp m$. \emph{In this paper, we use these coarse inequality notations only when $k,c$ depend only on the surface.} 
\end{notation}

Recall the following beautiful formula derived by Choi--Rafi:

\begin{theorem}[{\cite[Corollary D]{CR}}] \label{marking}
There exists $N$ such that for any markings $\sigma$ and $\tau$ on $S$, $$\log i(\sigma,\tau)\asymp \sum_{Z\subseteq S}[ d_{Z}(\sigma,\tau)]_{N}+\sum_{A\subseteq S} \log [d_{A}(\sigma,\tau)]_{N}$$ where $[m]_{n}=m$ if $m>n$, $[ m]_{n}=0$ if $m\leq n$, and the sum is taken over all $Z$ which are not annuli and $A$ which are annuli in $S$.
\end{theorem}

We show Theorem \ref{marking} for two curves $x,y\in C(S)$ where we have more freedom on cut--off constants, which is Theorem \ref{E}. 
We remark that, in \cite{W3}, the author showed that for any curves $x$ and $y$ on $S$, if $n>0$ then $$\log i(x,y)  \prec \sum_{Z\subseteq S}[ d_{Z}(x,y)]_{n}+\sum_{A\subseteq S} \log [d_{A}(x,y)]_{n}$$ deriving all quasi--constants by a different approach from \cite{CR}. Therefore, it is left to show the converse direction; we first start with $x,y\in C(S)$ and complete them into markings $\sigma,\tau$ controlling $i(\sigma,\tau)$ by $i(x,y)$, see Corollary \ref{E3}. Then we use Theorem \ref{marking} to obtain Theorem \ref{E}. 

\subsubsection{Constructing good markings from curves}
The goal of this subsection is to establish Corollary \ref{E3}, which follows from Lemma \ref{E1} and Lemma \ref{E2}. For completeness, we will keep track of most of constants which appear in the proofs of Lemma \ref{E1} and Lemma \ref{E2}. However, the efficient reader is welcome to skim through, taking note that these constants will depend only on the surface.

Suppose $A\subseteq S$. We let $S-A$ denote a ``single'' complementary component of $A$ in $S$ which is not a pair of pants. We note that this choice of the component will not cause any issues, i.e., we can take any component which is not a pair of pants as $S-A$.

We first observe the following for pants decompositions.

\begin{lemma}\label{E1}
Let $x,y\in C(S)$ such that $x$ and $y$ fill $S$. There exist pants decompositions $\sigma^{p}$ and $ \tau^{p}$ such that $x\in \sigma^{p}, y\in \tau^{p}$, and $i(\sigma^{p},\tau^{p})  \prec i(x,y).$
\end{lemma}
\begin{proof}
Let $x_{1}=x$ and $y_{1}=y$, we define $$x_{i+1}=x_{i} \cup \pi_{S-x_{i}}(y)\text{   and   }y_{i+1}=y_{i} \cup \pi_{S-y_{i}}(x).$$ We first note that $\pi_{S-x_{i}}(y)\neq \emptyset$ since $d_{S}(x_{i}, y)\geq d_{S}(x,y)-d_{S}(x,x_{i})>2-1=1,$ which implies $y$ essentially intersects with $S-x_{i}$. Similarly, $\pi_{S-y_{i}}(x)\neq \emptyset$. This is the only place where we use the fact that $x$ and $y$ fill $S$ so that $d_{S}(x,y)>2$.

This process terminates when $i=\x(S)-1$ since $x_{\x(S)}$ and $y_{\x(S)}$ are pants decompositions. 

We show
\begin{equation}
i(x_{i+1}, y_{i+1})\leq 9\cdot i(x_{i},y_{i})+ 4\cdot i(x,y). \tag{$\dagger$}
\end{equation}
To obtain $(\dagger)$, it suffices to show the following.
\begin{enumerate}
\item $ i(\pi_{S-x_{i}}(y),y_{i})\leq 2\cdot i(x_{i},y_{i}).$
\item $i(x_{i},\pi_{S-y_{i}}(x))\leq 2\cdot i(x_{i},y_{i}).$
\item $i(\pi_{S-x_{i}}(y), \pi_{S-y_{i}}(x))\leq 4\cdot i(x_{i},y_{i})+ 4\cdot i(x,y).$
\end{enumerate}

For the first inequality, we need to consider the intersections of $\pi_{S-x_{i}}(y)$ and $y_{i}$ only in the regular neighborhood of $\p(S-x_{i})$ since $i(y,y_{i})=0$. We observe these intersections are bounded by $2\cdot i(x_{i}, y_{i}).$ See Figure \ref{Fig1}.

\begin{figure}[h]
\centering
  \includegraphics[width=120mm]{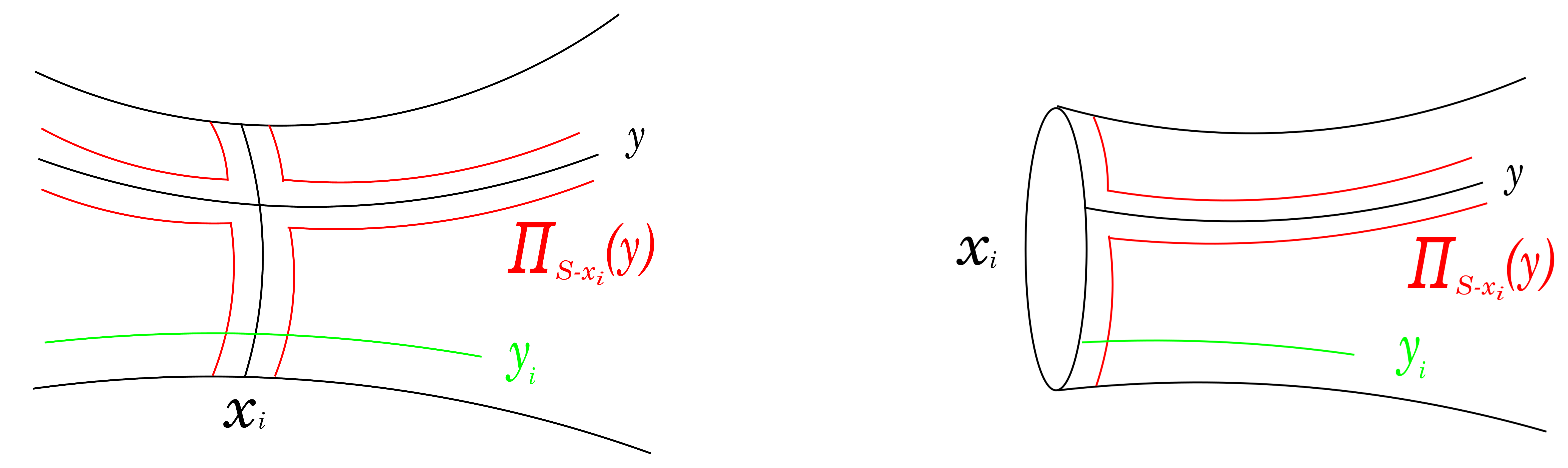}
 \caption{The left figure: $ i(\pi_{S-x_{i}}(y),y_{i})\leq 2\cdot i(x_{i},y_{i}).$ The right figure: if $x_{i}$ separates $S$, then we can sharpen so that $ i(\pi_{S-x_{i}}(y),y_{i})\leq i(x_{i},y_{i})$.}
 \label{Fig1}
\end{figure}

The same argument works to show the second inequality. 

For the third inequality, since $\pi_{S-x_{i}}(y)$ is contained in the regular neighborhood of $x_{i} \cup y$, it suffices to consider the intersections of $\pi_{S-x_{i}}(y)$ and $\pi_{S-y_{i}}(x)$ in the regular neighborhood of $x_{i}$ and in the regular neighborhood of $y$. 
\begin{itemize}
\item In the regular neighborhood of $x_{i}$, we can bound the intersections by $4\cdot i(x_{i},y_{i})$ by the definition of subsurface projections. See Figure \ref{Fig2} (Left). 

\item In the regular neighborhood of $y$, the intersections can arise only from the intersections of $x$ and $y$ since $i(y,y_{i})=0$; near every intersection of $x$ and $y$, $ \pi_{S-x_{i}}(y)$ and $\pi_{S-y_{i}}(x)$ intersect at most four times. See Figure \ref{Fig2} (Right). 
\end{itemize}
We have $i(\pi_{S-x_{i}}(y), \pi_{S-y_{i}}(x))\leq 4\cdot i(x_{i},y_{i})+ 4\cdot i(x,y).$
\begin{figure}[h]
\centering
 \includegraphics[width=120mm]{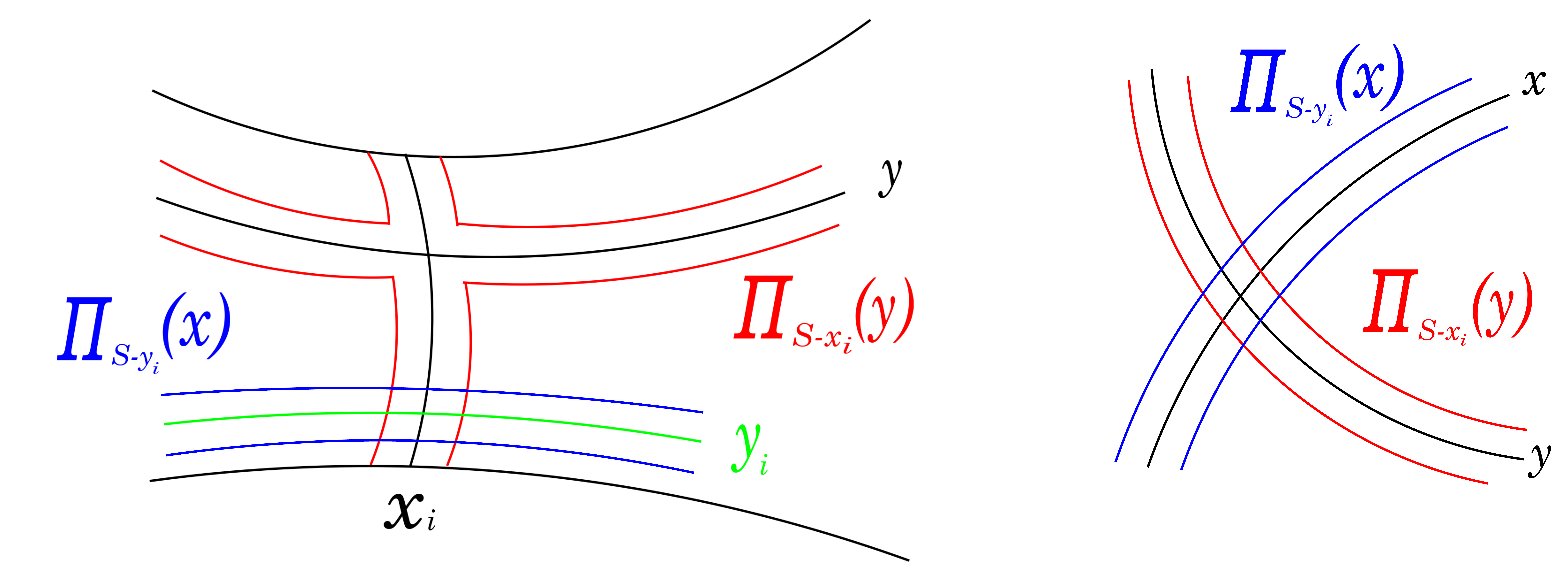}
 \caption{The left figure: $i(\pi_{S-x_{i}}(y), \pi_{S-y_{i}}(x))$ in the regular neighborhood of $x_{i}$. Note that if $x_{i}$ in the figure was $x$, then $\pi
 _{S-y_{i}}(x)$ would look different, but the same bound still works. The right figure: $i(\pi_{S-x_{i}}(y), \pi_{S-y_{i}}(x))$ in the regular neighborhood of $y$.}
 \label{Fig2}
\end{figure}

All together, we have ($\dagger$).
We let $\sigma^{p}=x_{\x(S)} \text{ and }\tau^{p}=y_{\x(S)}$. Then we have $$i(\sigma^{p},\tau^{p}) \prec i(x,y).$$
\end{proof}

We also observe the following for transversal curves. 

\begin{lemma}\label{E2}
Let $x,y\in C(S)$ such that $x$ and $y$ fill $S$, and let $\sigma^{p}$ and $\tau^{p}$ be pants decompositions such that $x\in \sigma^{p}$ and $y\in \tau^{p}$. There exist transversal curves $\sigma^{t}$ and $\tau^{t}$ so that by letting $\sigma=\sigma^{p}\cup \sigma^{t}$ and $\tau= \tau^{p}\cup \tau^{t} $ we have $i(\sigma, \tau)\prec i(\sigma^{p},\tau^{p})+i(x,y).$
\end{lemma}
\begin{proof}

We prove the statement by the following steps. Throughout, we use similar arguments given in Lemma \ref{E1}.

\underline{\textbf{Step 1 (Construction of $\sigma^{t}$ for $\sigma^{p}$):}}
For each curve $a \in \sigma^{p}$, we find a transversal curve $a^{t}$. 
Let $W\subseteq S$ such that $\x(W)=1$, $a\in C(W)$, and $\p(W) \subseteq \{ \sigma^{p}\cup \partial(S)\}.$ We take $a^{t}=\pi_{W}(y)\in C(W)$, note that $a^{t}$ exists because $x$ and $y$ fill $S$.

Since $i(y, \tau^{p})=0$, we have $i(a^{t}, \tau^{p} )\leq 2\cdot  i( \p(W),\tau^{p}).$ 
Now, since $\p(W)\subseteq \sigma ^{p}$, we have $i(a^{t}, \tau^{p} )\leq  2\cdot i( \p(W),\tau^{p}) \leq 2\cdot  i(\sigma^{p},\tau^{p}).$

We do this process for every curve in $\sigma^{p}$ and obtain the set of transversal curves $\sigma^{t}$. Then we have 
\begin{equation}
i(\sigma^{t},\tau^{p})  \leq 2\x(S)\cdot i(\sigma^{p},\tau^{p}).
\end{equation}

We also make the following observation for the next step.
For any $a\in \sigma ^{p}$ we have either $x\notin C(W)$ or $x\in C(W)$. In the first case, $i(a^{t},x)=0$. In the second case, near every intersection of $x$ and $y$, we see that $x$ and $a^{t}$ intersect at most twice; so $i(a^{t},x)\leq 2\cdot i(x,y).$ Since $|\sigma^{t}|=\x(S)$, we have 
\begin{equation}
i(\sigma^{t},x)\leq \x(S)\cdot (2\cdot i(x,y))= 2\x(S) \cdot i(x,y). \tag{$\ddagger$}
\end{equation}

\underline{\textbf{Step 2 (Construction of $\tau^{t}$ for $\tau^{p}$):}}
For each curve $b \in \tau^{p}$, we find a transversal curve $b^{t}$. 
Let $V\subseteq S$ such that $\x(V)=1$, $b\in C(V)$ and $\p(V) \subseteq \{ \tau^{p}\cup \partial(S)\}.$ 
We take $b^{t}=\pi_{V}(x)\in C(V).$ 

We first observe (\rmnum{1}) and (\rmnum{2}) to show $i(\sigma,b^{t})\leq  6\x(S) \cdot \big(i(\sigma^{p},\tau^{p})+ i(x,y)\big).$
\begin{enumerate}[label=(\roman*)]
\item $i(\sigma^{p}, b^{t} ) \leq 2\cdot i(\sigma^{p}, \tau^{p}).$
\item $i(\sigma^{t},b^{t}) \leq 4\x(S)\cdot i(\sigma^{p},\tau^{p})+4\x(S)\cdot i(x,y).$
\end{enumerate}

For (\rmnum{1}), we use the same argument given in the previous step; we have $i(\sigma^{p}, b^{t} )\leq 2\cdot i(\sigma^{p}, \p(V)) \leq 2\cdot i(\sigma^{p}, \tau^{p}).$

For (\rmnum{2}), we consider the intersections of $\sigma^{t}$ and $b^{t}$ in the regular neighborhood of $\p(V)$ and its complementary component in $V$.
See Figure \ref{Fig4}. 
We have 
\begin{align}
   i(\sigma^{t},b^{t})&\leq 2\cdot i(\sigma^{t},\p(V))+2\cdot i(\sigma^{t},x)\tag*{}\\
   	&  \leq 2\cdot i(\sigma^{t},\tau^{p})+2\cdot i(\sigma^{t},x) \tag{Since $\p(V)\subseteq \tau^{p}$}\\
                  & \leq  2\cdot i(\sigma^{t},\tau^{p})+4\x(S) \cdot i(x,y) \tag{By ($\ddagger$)}\\
                  & \leq   4\x(S)\cdot i(\sigma^{p},\tau^{p})+4\x(S) \cdot i(x,y). \tag{By (1)}
                  \end{align}

\begin{figure}[hbtp]
\centering
   \includegraphics[width=120mm]{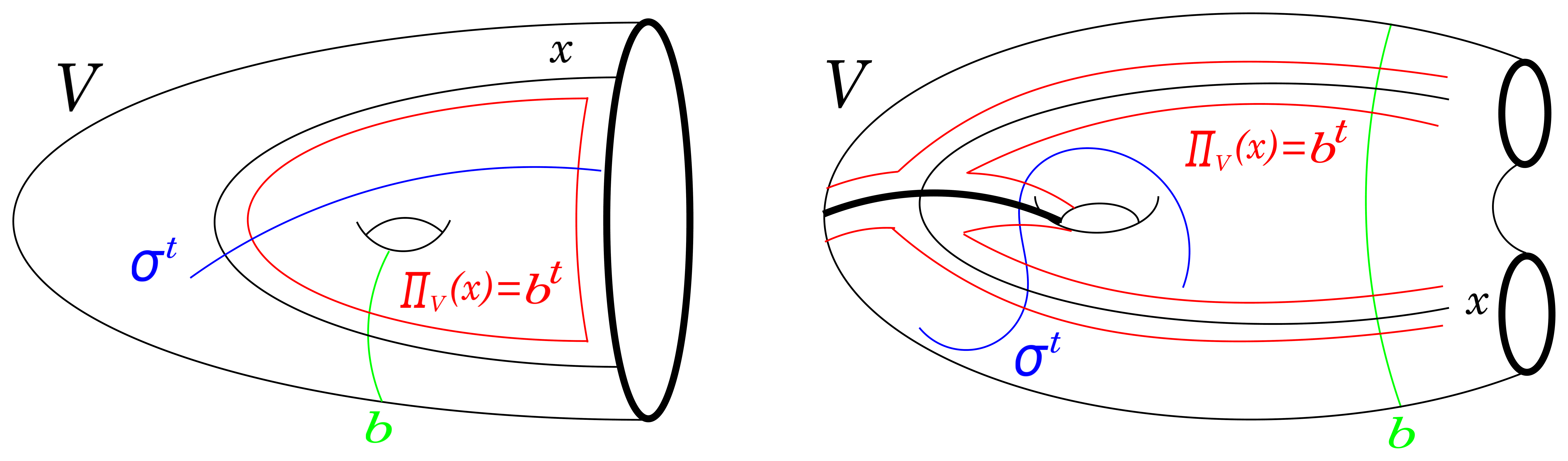}
 \caption{Bold lines represent $\partial(V)\subseteq \{ \tau^{p}\cup \partial(S)\}$. We observe $i(\sigma^{t},b^{t})$ in the regular neighborhood of $\p(V)$ is bounded by $2\cdot i(\sigma^{t},\p(V))$ and $i(\sigma^{t},b^{t})$ in the complement of the regular neighborhood of $\p(V)$ is bounded by $2\cdot i(\sigma^{t},x)$.}
 \label{Fig4}
\end{figure}

Therefore, we have 
\begin{align}
   i(\sigma,b^{t})&\leq i(\sigma^{p},b^{t})+i(\sigma^{t},b^{t})\tag{Since $\sigma=\sigma^{p}\cup \sigma^{t}$}\\
   	&  \leq 2\cdot i(\sigma^{p},\tau^{p})+4\x(S)\cdot i(\sigma^{p},\tau^{p})+4\x(S) \cdot i(x,y) \tag{By (\rmnum{1}) and (\rmnum{2})}\\
	&\leq  6\x(S) \cdot \big(i(\sigma^{p},\tau^{p})+ i(x,y)\big).\tag*{}
                 \end{align}

We do this process for every curve in $\tau^{p}$ and obtain the set of transversal curves $\tau^{t},$ and we have 
\begin{equation}
i(\sigma, \tau^{t}) \leq 6\x(S)^{2} \cdot \big(i(\sigma^{p},\tau^{p})+ i(x,y)\big).
\end{equation}
\underline{\textbf{Step3 (Checking $i(\sigma, \tau)\prec i(\sigma^{p},\tau^{p})+i(x,y)$):}}
Lastly, we take $\tau= \tau^{p}\cup \tau^{t},$ then we have 
\begin{align}
   i(\sigma, \tau) &=i(\sigma, \tau^{p})+i(\sigma, \tau^{t}) \tag{Since $\tau=\tau^{p}\cup \tau^{t}$}\\
   	& =i(\sigma^{p}, \tau^{p})+i(\sigma^{t}, \tau^{p})+i(\sigma, \tau^{t}) \tag{Since $\sigma=\sigma^{p}\cup \sigma^{t}$}\\
                  &\prec i(\sigma^{p},\tau^{p})+ i(x,y).\tag{By (1) and (2)}
\end{align}

\end{proof}

By Lemma \ref{E1} and Lemma \ref{E2}, we have 

\begin{corollary}\label{E3}
Let $x,y\in C(S)$ such that $x$ and $y$ fill $S$. There exist markings $\sigma$ and $\tau$ such that $x\in \sigma$, $y\in \tau$, and $i(\sigma, \tau)\prec i(x,y).$
\end{corollary}
\begin{proof}
Let $\sigma$ and $\tau$ be the markings given by Lemma \ref{E1} and \ref{E2}. We have
\begin{align}
   i(\sigma, \tau) &\prec i(\sigma^{p},\tau^{p})+ i(x,y) \tag{By Lemma \ref{E2}}\\
              & \prec i(x,y). \tag{By Lemma \ref{E1}}
\end{align}
\end{proof}

\subsubsection{Choi--Rafi formula for two curves}

We observe
\begin{theorem}\label{E}
There exists $N$ such that the following holds for any curves $x$ and $y$ on $S$; if $n\geq N$ then $$\log i(x,y) \asymp \sum_{Z\subseteq S}[ d_{Z}(x,y)]_{n}+\sum_{A\subseteq S} \log [d_{A}(x,y)]_{n}.$$ \end{theorem}

\begin{proof}
\underline{\textbf{If $x$ and $y$ fill $S$:}}
By Corollary \ref{E3}, there exist markings $\sigma$ and $\tau$ such that $x\in \sigma, y\in \tau$, and $\log i(\sigma, \tau)\prec \log i(x,y).$ We have 
\begin{align}
   \sum_{Z\subseteq S}[ d_{Z}(x,y)]_{n}+\sum_{A\subseteq S} \log [d_{A}(x,y)]_{n} &\leq \sum_{Z\subseteq S}[ d_{Z}(\sigma,\tau)]_{n}+\sum_{A\subseteq S} \log [d_{A}(\sigma,\tau)]_{n} \tag*{}\\
   &\leq \sum_{Z\subseteq S}[ d_{Z}(\sigma,\tau)]_{N}+\sum_{A\subseteq S} \log [d_{A}(\sigma,\tau)]_{N} \tag{Since $n\geq N$}\\
& \prec \log i(\sigma, \tau) \tag{By Theorem \ref{marking}}\\
              & \prec \log i(x,y). \tag{By Corollary \ref{E3}}
\end{align}


\underline{\textbf{If $x$ and $y$ do not fill $S$:}}
We take $F(x,y)\subset S$, then $x$ and $y$ fill $F(x,y)$. By the same argument in the previous case, we have $$\sum_{Z\subseteq F(x,y)}[ d_{Z}(x,y)]_{n}+\sum_{A\subseteq F(x,y)} \log [d_{A}(x,y)]_{n} \prec \log i(x,y).$$ 
We note that $Z$ and $A$ on the above formula need to range over the whole surface for the statement of this theorem. However, if $W\subseteq S$ such that $W \nsubseteq F(x,y)$, $\pi_{W}(x)\neq \emptyset$ and $\pi_{W}(y) \neq \emptyset$, then by Lemma \ref{oct} we have $$d_{W}(x,y)\leq d_{W}(x,\p(F(x,y)))  +d_{W}(\p(F(x,y)),y)\leq 2+2.$$ 
By taking $N\geq 4$ if necessary, we have $[ d_{W}(x,y)]_{n}=0.$

\end{proof}

In the rest of this paper, $N$ will denote the constant given by Theorem \ref{E}.


\section{The proofs} 
We prove Theorem \ref{E'} and Theorem \ref{E''} in a coarse inequality setting. 

\begin{theorem}
Let $x,y\in C(S)$ and $g_{x,y} = \{v_{i}\}$ be a tight geodesic such that $d_{S}(x,v_{i}) = i$ for all $i$. We have
$\log i(v_{p},v_{q}) \prec \log i(x,y)$ for all $p,q$. 

\end{theorem}
\begin{proof} 
The proof is the combination of Lemma \ref{sss} and Theorem \ref{E}. 
Assume $p<q$. Take $k$ such that $k\geq N+2M$.
If $W$ is a proper subsurface such that $[d_{W}(v_{p},v_{q})]_{k}>0$, then $d_{W}(x, v_{p})\leq M$ $d_{W}(v_{q},y)\leq M$ by Lemma \ref{sss}. Therefore, we have $$ d_{W}(x,y)\geq d_{W}(v_{p}, v_{q})- d_{W}(x,v_{p})-d_{W}(v_{q},y)\geq d_{W}(v_{p}, v_{q})-2M ;$$ in particular we have $[d_{W}(x,y)]_{k-2M}>0$. By taking larger $k$ if necessary, so that $k\leq 2\cdot (k-2M)$ and $k\leq (k-2M)^{2}$, we have 
\begin{itemize}
\item $[d_{W}(v_{p},v_{q})]_{k}\leq 2\cdot [d_{W}(x,y)]_{k-2M}.$
\item $\log [d_{W}(v_{p},v_{q})]_{k}\leq 2\cdot \log [d_{W}(x,y)]_{k-2M}. $
\end{itemize}
Furthermore, since we clearly have $[d_{S}(v_{p},v_{q})]_{k} \leq [d_{S}(x,y)]_{k-2M}$; all together we obtain $$\displaystyle \sum_{Z\subseteq S}[ d_{Z}(v_{p},v_{q})]_{k}+\sum_{A\subseteq S} \log [d_{A}(v_{p},v_{q})]_{k} \leq 2 \cdot \bigg(\sum_{Z\subseteq S}[ d_{Z}(x,y)]_{k-2M}+\sum_{A\subseteq S} \log [d_{A}(x,y)]_{k-2M}\bigg).$$

Lastly, by our choice of $k\geq N+2M$, we can apply Theorem \ref{E} to the above, and we have 
\begin{align}
   \log i(v_{p},v_{q})& \prec \sum_{Z\subseteq S}[ d_{Z}(v_{p},v_{q})]_{k}+\sum_{A\subseteq S} \log [d_{A}(v_{p},v_{q})]_{k} \tag*{}\\
   	&  \leq  2 \cdot \bigg(\sum_{Z\subseteq S}[ d_{Z}(x,y)]_{k-2M}+\sum_{A\subseteq S} \log [d_{A}(x,y)]_{k-2M}\bigg) \tag*{}\\
	&\prec \log i(x,y).\tag*{}
                 \end{align}

\end{proof}

By using a similar technique in the proof of the above theorem, we show the following theorem. Note that we will not require a geodesic in the statement to be tight.

\begin{theorem}
Let $x,y\in C(S)$ and $g_{x,y} = \{v_{i}\}$ be a geodesic such that $d_{S}(x,v_{i}) = i$ for all $i$. We have $ \log i(x,y)\prec  \log \big( i(x,v_{p})\cdot i(v_{p},y)\cdot i(x,v_{q})\cdot i(v_{q},y)\big)$ for all $p,q$ such that $|p-q|>2$. (We let $i(x,v_{1})=1$ and $i(v_{d_{S}(x,y)-1},y)=1$.)
\end{theorem}
\begin{proof} 

Let $v_{i}\in g_{x,y}$; we define $\W^{i}:=\{W\subseteq S| \pi_{W}(x)\neq \emptyset, \pi_{W}(y)\neq \emptyset, \pi_{W}(v_{i})\neq \emptyset \}.$

Take $k\geq 2\cdot N$ and let $l =\big\lfloor \frac{k}{2} \big\rfloor $.
Let $W\in \W^{p}$. If $[d_{W}(x,y)]_{k}>0$ then $[d_{W}(x,v_{p})]_{l}>0 \text{ or }[d_{W}(v_{p},y)]_{l}>0.$ Therefore, taking larger $k$ if necessary, we have 
\begin{itemize}
\item $[d_{W}(x,y)]_{k}\leq 2\cdot \big( [d_{W}(x,v_{p})]_{l} +[d_{W}(v_{p},y)]_{l}\big).$
\item $\log [d_{W}(x,y)]_{k}\leq 2\cdot \big(\log [d_{W}(x,v_{p})]_{l} +\log[d_{W}(v_{p},y)]_{l}\big).$
\end{itemize}
Thus, we have 
\begin{eqnarray*}
\sum_{Z \in \W^{p} }[ d_{Z}(x,y)]_{k}+\sum_{A \in \W^{p} } \log [d_{A}(x,y)]_{k}&\leq& 2\cdot \bigg(\sum_{Z \in \W^{p}}[ d_{Z}(x,v_{p})]_{l}+\sum_{A \in \W^{p} } \log [d_{A}(x,v_{p})]_{l} \bigg)\\&+&2\cdot \bigg(\sum_{Z \in \W^{p} }[ d_{Z}(v_{p},y)]_{l}+\sum_{A \in \W^{p}  } \log [d_{A}(v_{p},y)]_{l} \bigg).
\end{eqnarray*}

Lastly, we notice that every subsurface of $S$, to which $x$ and $y$ project nontrivially, is contained in $\W^{p} \cup \W^{q}$ because $v_{p}$ and $v_{q}$ fill $S$.
We repeat the same argument on $\W^{q}$, and combining with the above observation on $\W^{p}$, we have
\begin{eqnarray*}
\sum_{Z \subseteq S  }[ d_{Z}(x,y)]_{k}+\sum_{A \subseteq S } \log [d_{A}(x,y)]_{k} &\leq& 2\cdot \bigg(\sum_{Z\subseteq S }[ d_{Z}(x,v_{p})]_{l}+\sum_{A\subseteq S } \log [d_{A}(x,v_{p})]_{l} \bigg)\\&+&2\cdot \bigg(\sum_{Z \subseteq S }[ d_{Z}(v_{p},y)]_{l}+\sum_{A  \subseteq S } \log [d_{A}(v_{p},y)]_{l} \bigg)\\&+&2\cdot \bigg(\sum_{Z \subseteq S}[ d_{Z}(x,v_{q})]_{l}+\sum_{A\subseteq S } \log [d_{A}(x,v_{q})]_{l} \bigg)\\&+&2\cdot \bigg(\sum_{Z \subseteq S }[ d_{Z}(v_{q},y)]_{l}+\sum_{A \subseteq S  } \log [d_{A}(v_{q},y)]_{l} \bigg).
\end{eqnarray*}

Since $l =\big\lfloor \frac{k}{2} \big\rfloor \geq N$, we can apply Theorem \ref{E} to the above to obtain $$ \log i(x,y)\prec  \log i(x,v_{p})+ \log i(v_{p},y)+\log i(x,v_{q})+\log i(v_{q},y).$$
\end{proof}

\begin{remark}\label{hkai}
We let $i(x,v_{1})=1$ even though $i(x,v_{1})=0$ (Similarly, we let $i(v_{d_{S}(x,y)-1},y)=1$.) in the statement of the above theorem because if $W\in \W^{1}$ then $[d_{W}(x,v_{1})]_{l} \leq [2]_{l}=0$ by Lemma \ref{oct}, i.e., we have $$\sum_{Z \in \W^{1} }[ d_{Z}(x,y)]_{k}+\sum_{A \in \W^{1} } \log [d_{A}(x,y)]_{k} \leq 2\cdot \bigg(\sum_{Z \in \W^{1} }[ d_{Z}(v_{1},y)]_{l}+\sum_{A \in \W^{1}  } \log [d_{A}(v_{1},y)]_{l} \bigg).$$
\end{remark}



\bibliographystyle{plain}
\bibliography{references.bib}

\end{document}